\documentclass[12pt]{amsart}
\usepackage[margin=1in]{geometry}
\usepackage{amsmath,amsfonts,graphicx}
\usepackage[utf8]{inputenc}
\usepackage{amsthm}
\usepackage{amssymb}
\usepackage{enumerate}
\usepackage{tikz-cd}
\usepackage{float}
\usepackage{hyperref}
\usepackage[capitalise]{cleveref}
\usepackage[utf8]{inputenc}
\usepackage[english]{babel}
\usepackage{lipsum}
\graphicspath{{./images/}}

\title{From Ideal Membership Problem for polynomial rings to Dehn Functions of Metabelian Groups}
\author{Wenhao Wang}
\address{Department of Mathematical Logic\\
  The Steklov Mathematical Institute of Russian Academy of Science\\
 Moscow, Russia 119991}
\email[W.~Wang]{wenhaowang@mi-ras.ru}

\newtheorem{theorem}{Theorem}[section]

\newtheorem{corollary}[theorem]{Corollary}

\newtheorem*{imp}{Ideal Membership Problem}
\newtheorem*{rp}{Representation Problem}
\newtheorem{innercustomthm}{Theorem}
\newenvironment{customthm}[1]
  {\renewcommand\theinnercustomthm{#1}\innercustomthm}
  {\endinnercustomthm}

\theoremstyle{definition}
\newtheorem{proposition}[theorem]{Proposition}
\newtheorem{definition}[theorem]{Definition}

\newtheorem{prob}[theorem]{Problem}

\theoremstyle{remark}

\newtheorem*{rems}{Remark}

\newcommand{\llangle}{\langle \langle}
\newcommand{\rrangle}{\rangle \rangle}

\DeclareMathOperator{\area}{Area}

\DeclareMathOperator{\supp}{supp}

\begin{document}

\maketitle

\begin{abstract}
The ideal membership problem asks whether an element in the ring belongs to the given ideal. In this paper, we propose a function that reflecting the complexity of the ideal membership problem in the ring of Laurent polynomials with integer coefficients. We also connect the complexity function we define to the Dehn function of a metabelian group, in the hope of constructing a metabelian group with superexponential Dehn function.
\end{abstract}

\section{Introduction}

Let $R$ be a commutative ring with unity. Let $I$ be an ideal of $R$ along with a generating set $F=\{f_1,f_2,\dots,f_s\}$. The following problems are well known. 
\begin{imp}
	Let $I$ be an ideal of the ring $R$. Decide whether a given element $g\in R$ belongs to $I$.
\end{imp}

\begin{rp}
		Let $I$ be a finitely generated ideal of the ring $R$ and $F=\{f_1,f_2,\dots,f_s\}$ be a generating set of $I$. Given an element $g\in I$, find $h_i\in R$ such that
	\[g=h_1 f_1+h_2 f_2+\dots+f_s h_s.\]

\end{rp}

In this paper, we fix $R$ to be the polynomial ring or the ring of Laurent polynomials over $\mathbb Z$, i.e., $R=\mathbb Z[x_1,x_2,\dots,x_n]$ or $\mathbb Z[x_1^{\pm 1},x_2^{\pm 1},\dots,x_k^{\pm 1}]$. Both problems are decidable in these cases (see in \cite{mayr1997some} or \cite{mayr2017complexity}). Our goal is to establish a function reflecting the computational complexity of these two problems. Note that since $R$ is Noetherian, every ideal is finitely generated.

For $d=(d_1,d_2,\dots,d_k)\in \mathbb Z^k$ or $\mathbb Z_{\geqslant 0}^k$ we will denote the term $x_1^{d_1}x_2^{d_2} \dots x_k^{d_k}$ by $x^d$. Additionally, we equip $\mathbb Z^k$ with the norm $\|d\|:=|d_1|+|d_2|+\dots+|d_k|$. An element $g\in R$ can be uniquely written as 
$g=\sum_{i=1}^l \alpha_{d_i} x^{d_i}, \alpha_{d_i}\in \mathbb Z, d_i\in \mathbb Z^k$ or $\mathbb Z_{\geqslant 0}^k.$ We then define the norm and degree of $g$ as $|g|=\sum_{i=1}^l |\alpha_i|,\deg g=\max_{i} \|d_i\|$ respectively.

We define a function that gives the minimal number of generators in $F$ needed to express an element $g\in I.$ Analogously to the case of groups, we call it the area of $g$ with respect to $I$ and $F$.
\[\area_{I,F}(g)=\min \{\sum_{i=1}^s |h_i| \mid g=h_1 f_1+h_2f_2+\dots+h_sf_s, h_i\in R\}.\]

We then define the \emph{complexity function} of $A$ with respect to $F$ to be 
\[C_{I,F}(m,n):=\sup\{\area_F(g) \mid |g|\leqslant m, \deg g \leqslant n\}.\]

%We will regard complexity functions up to the equivalence relation $\asymp$, where functions $f,g: \mathbb N\times \mathbb N \to \mathbb N$ are equivalent if there exists a constant $C>0$ such that $C^{-1}f\leqslant g \leqslant Cf$ for all $\mathbb N\times \mathbb N$.

The first result is the following:
\begin{customthm}{A}[\cref{degreeBound}]
	\label{degreeBoundIntro}
	Let $A$ be an ideal in the polynomial ring $R=\mathbb Z[x_1,x_2,\dots,x_k]$ or $\mathbb Z[x_1^{\pm 1},x_2^{\pm 1},\dots,x_k^{\pm 1}]$ generated by $F=\{f_1,f_2,\dots,f_s\}$. Let $g\in A$ and $h_i\in R$ such that 
	\[g=h_1 f_1+h_2 f_2+\dots +h_s f_s,\  \area_{I,F}(g)=\sum_{i=1}^s |h_i|.\]
	Then there exists a constant $C>0$ depending solely on $F$, such that
	\[\deg h_i \leqslant \deg g + C\area_{I,F}(g).\]
\end{customthm}

In the literature, most results concerning the ideal membership problem for the ring of polynomials estimate the degree of the coefficients \cite{hermann1926PolynomialIdeal}, \cite{mayr1982complexity}. \cref{degreeBoundIntro} provides a mathod to translate the estimation of the degree to an estimation of the area, and hence to the complexity function.  

The complexity function of an ideal we propose is an analog of the Dehn function of a finitely presented group. Given a finitely presented group $G$ along with a finite presentation $\langle X \mid R\rangle$. The Dehn function $\delta_G(n)$ describes how many relations in $R$ needed to represented a word that is trivial in $G$ and of length at most $n$ in $X$ (formal definition of Dehn function will be given in Section 3). It is well-known that the word problem of a finitely presented group is decidable if and only if its Dehn function is computable \cite{MadlenerOtto1985}.  

In \cite{wang2020dehn}, it is shown that Dehn functions of finitely presented metabelian groups share a universal upper bound $2^{f(n)}$, where $f(n)$ is any superpolynomial function. The answer to the following question remains unknown.

\begin{prob}
\label{question01}
	Is the (relative) Dehn function of any finitely presented (generated) metabelian group bounded above by the exponential function?
\end{prob}

All known examples to this date are bounded above by the exponential function. The next result connects the complexity function of an ideal to the Dehn function of a finitely presented metabelian group, providing a potential method that translates a complicated ideal membership problem into a large Dehn function of a finitely presented metabelian group.

\begin{customthm}{B}[\cref{relationship}, \cref{finitelyPresented}]
	\label{relationshipIntro}
	Let $I$ be an ideal in the ring $\mathbb Z[x_1^{\pm 1},x_2^{\pm 1},\dots,x_k^{\pm 1}]$. Then there exists a finitely generated metabelian group $G_I$ with a presentation $\mathcal P$ and a constant $C>0$ such that
	\[\tilde \delta_{G_I,\mathcal P}(C\cdot mn) \geqslant C_I(m,n),\]
	where $\tilde \delta_G$ be the Dehn function relative to the variety of metabelian group. In particular, if $I$ is tame, then there exists a finite presentation $\mathcal P'$ of $G_I$ and a constant $C'>0$ such that 
	\[C'\cdot \delta_{G_I,\mathcal P'}(C' mn)+C'mn+C'\geqslant  C_I(m,n).\]
\end{customthm}

The definitions of the relative Dehn function and tameness of an ideal will be given in Section 3.

\emph{Acknowledgements.} The author acknowledges Igor Lysenok for many valuable discussions on this subject. The work was performed at the Steklov International Mathematical Center and supported by the Ministry of Science and Higher Education of the Russian Federation (agreement no. 075-15-2022-265).

\section{The Complexity Function of an Ideal}

As in previous section, we fix $R$ to be $R=\mathbb Z[x_1,x_2,\dots,x_k]$ or $\mathbb Z[x_1^{\pm 1},x_2^{\pm 1},\dots,x_k^{\pm 1}]$. Let $I$ be an ideal in $R$ generated by $f_1,f_2,\dots,f_s.$

We define the equivalence relation on functions $f,g: \mathbb N\times \mathbb N \to \mathbb N$ as follows
\[f(m,n) \asymp g(m,n) \iff \exists C\in \mathbb N, C^{-1}f(m,n)\leqslant g(m,n)\leqslant Cf(m,n).\]

\begin{proposition}
	The complexity function is unique with respect to any finite generating set up to the equivalence relation $\asymp$.
\end{proposition}

\begin{proof}
	Let $F'=\{g_1,g_2,\dots,g_l\}$ be another generating set of $I$. Then every $f_i$ can be decomposed to a linear cobination of $g_1,g_2,\dots,g_l$ with coefficients in $R$:
	\[f_i=\sum_{j=1}^l r_{i,j}g_j.\]
	For an arbitrary $g\in A$, there exists $h_1,h_2,\dots h_s$ such that 
	\[g=h_1 f_1 + h_2 f_2 +\dots+h_s f_s, \text{ and } \area_{I,F}=\sum_{i=1}^s |h_i|.\]
	Then we have
	\[g=\sum_{i=1}^s \left( h_i\sum_{j=1}^{l}  r_{i,j}g_j\right)=\sum_{j=1}^l \left( g_j \sum_{i=1}^s h_i r_{i,j} \right).\]
	Let $C=\max_{i,j} |r_{i,j}|$. It follows that 
	\[\area_{I,F'}(g)\leqslant \sum_{j=1}^l \left|\sum_{i=1}^s h_i r_{i,j}\right|\leqslant \sum_{j=1}^l C\sum_{i=1}^s |h_{i}|\leqslant lC\area_{I,F}(g).\]
	Therefore, the proposition follows.
\end{proof}

Thus, there is no ambiguity in speaking about the complexity function of the ideal $I$, and we will denote it by $C_I(m,n)$.

With the help of the classic Gr\"{o}bner basis method to solve the ideal membership problem for $I$, it is easy to check that 

\begin{proposition}
	Let $I$ be an ideal in the ring $R$. 
	If $R=\mathbb Z[x_1,x_2,\dots,x_k]$, we have that 
	\[C_I(m,n)\leqslant C m^{n^{k}},\]
	for some constant $C>0$.
	If $R=\mathbb Z[x_1^{\pm 1},x_2^{\pm 1},\dots,x_k^{\pm 1}]$, we have that 
	\[C_I(m,n)\leqslant D m^{n^{2k}},\]
	for some constant $D>0$.
\end{proposition}

\begin{proof}
	Choose a Gr\"{o}bner basis of $I$, then the result follows immediately. For details, we refer to \cite{wang2020dehn}. 
\end{proof}

By the support of $g=\sum_{i=1}^l \alpha_{d_i} x^{d_i}, \alpha_{d_i}\in \mathbb Z, d_i\in \mathbb Z^k$ or $\mathbb Z_{\geqslant 0}^k$, we mean 
\[\supp(g)=\{d_1,d_2,\dots,d_l\}\subset \mathbb Z^k \text{ or }\mathbb Z_{\geqslant 0}^k.\]

\begin{theorem}
\label{degreeBound}
	Let $I$ be an ideal generated by $F=\{f_1,f_2,\dots,f_s\}$. Let $g\in I$ and $h_i\in R$ such that 
	\[g=h_1 f_1+h_2 f_2+\dots +h_s f_s,\  \area_{I,F}(g)=\sum_{i=1}^s |h_i|.\]
	Then there exists a constant $C>0$ depending solely on $F$, such that
	\[\deg h_i \leqslant \deg g + C\area_{I,F}(g).\]
\end{theorem}

\begin{proof}
	Each $h_i$ can be written as
	\[h_i=\sum_{j=1}^{|h_i|} \varepsilon_{d_{i,j}} x^{d_{i,j}}, \varepsilon_{d_{i,j}}\in \{\pm 1\}, d_{i,j}\in \mathbb Z^k.\]
	WLOG we assume that $x^{d_{1,1}}$ is the term of the maximal degree among $h_1,h_2,\dots,h_s$, i.e.,
	\[\|d_{1,1}\|=\max \{\deg h_1,\deg h_2,\dots, \deg h_s\}. \]
	We denote $J=\{ (i,j) \mid 1\leqslant i \leqslant s, 1\leqslant j \leqslant |h_i| \}.$ Note that $\area(g)=|J|$.
	
	To simplify the notation, we will denote $\varepsilon_{i,j} x^{d_{i,j}}f_{i}$ by $t_{i,j}$. By a partial sum of $h_1f_1+h_2f_2+\dots h_s f_s$ defined by a subset $K$ of $J$, we mean 
	\[\sum_{\mathbf{j}\in K} t_{\mathbf{j}}. \]
	
	A partial sum is called consecutive if it can be written in the following way
	\[t_{\mathbf{j}(1)}+t_{\mathbf{j}(2)}+\dots+t_{\mathbf{j}(l)}\]
	such that 
	\[\supp t_{\mathbf{j}(p+1)} \cap \supp \sum_{i=1}^{p} t_{\mathbf{j}(i)} \neq \emptyset, \text{ for } p=1,2,\dots,l-1.\]
	
	We claim that there exists a consecutive partial sum 
	\[r:= t_{\mathbf{j}(1)}+t_{\mathbf{j}(2)}+\dots+t_{\mathbf{j}(l)}\]
	satisfying $\mathbf{j}(1)=(1,1)$ and $\supp r \cap \supp g\neq \emptyset$.
	
	If such a partial sum does not exist, then every consecutive partial sum starting with $t_{(1,1)}$ will have no common term with $g$. Let us take an arbitrary consecutive partial sum starting with $t_{1,1}$
	\[r_0=t_{\mathbf{j}(1)}+t_{\mathbf{j}(2)}+\dots+t_{\mathbf{j}(l')}, \mathbf{j}(1)=(1,1).\]
	If $r_0\neq 0$, note that $g-r_0$ defines another partial sum, 
	\[g-r_0 = \sum_{\mathbf{j}\in K'} t_{\mathbf{j}}, K'=J\setminus \{\mathbf j(1),\mathbf j(2),\dots,\mathbf j(l')\}\]
	Since $\supp g\ \cap\  \supp r_0 = \emptyset$ by our assumption and $g=g-r_0+r_0$, then
	 \[\left( \cup_{\mathbf{j}\in K'} \supp t_{\mathbf{j}} \right)\cap \supp(r_0)\neq \emptyset\] 
	 
	 Thus, we take any term in the intersection and add it to $r_0$, i.e.,
	\[r_1:=r_0+t_{\mathbf{j}(l'+1)}, \supp t_{\mathbf{j}(l'+1)} \in \left( \cup_{\mathbf{j}\in K'} \supp t_{\mathbf{j}} \right)\cap \supp(r_0).\]
	It follows that we can extend a non-zero consecutive partial sum starting with $t_{1,1}$. Then the maximal consecutive partial sum will end up being 0, since we only have finitely many terms to exhaust. Therefore, we obtain a partial sum defined by a subset $K_0$ such that 
	\[\sum_{\mathbf{j}\in K_0} t_{\mathbf{j}}=0.\]
	Hence, we have
	\[g= \sum_{\mathbf{j}\in K_0} t_{\mathbf{j}}+ \sum_{\mathbf{j}\notin K_0} t_{\mathbf{j}}= \sum_{\mathbf{j}\notin K_0} t_{\mathbf{j}},\]
	which leads to a contradiction, since $\area(g)\leqslant |K_0|<|J|$. The claim is proved.
	
	We denote $C=\max \{\mathrm{diag}(f_i) \mid i=1,2,\dots,s\}$, where the diagonal of a subset in $\mathbb Z^k$ is giving by the canonical way with respect to the norm $\| \cdot \|$. One immediate observation is that $\mathrm{diag}(\supp t_{\mathbf i})\leqslant C$, since translation does not change the diagonal. Then we have 
	\[\mathrm{dist}(\supp g, \supp (t_{\mathbf{i}}+t_{\mathbf{j}})\geqslant \mathrm{dist}(\supp g, \supp t_{\mathbf{i}})-C, \text{ where } \supp t_{\mathbf i} \cap \supp t_{\mathbf{j}}\neq \emptyset.\]
	Because since $\supp t_{\mathbf i} \cap \supp t_{\mathbf{j}}\neq \emptyset $, then 
	\[\mathrm{dist}(t,\supp  t_{\mathbf i})\leqslant C, \forall t\in \supp t_{\mathbf j}.\]
	
	Finally, we have 
	\[\mathrm{dist}(\supp g,\supp r)\geqslant \mathrm{dist}(\supp g, \supp t_{\mathbf{j}(1)}) - lC \geqslant \mathrm{dist}(\supp g, \supp t_{\mathbf{j}(1)})-C\area_{I,F}(g).\]
	Thus
	\[C\area_{I,F}(g)\geqslant \mathrm{dist}(\supp g, \supp x^{d_{(1,1)}}f_1).\]
	Notice that by triangle inequalities
	\[\mathrm{dist}(\supp g, \supp x^{d_{(1,1)}}f_1)\geqslant \deg x^{d_{(1,1)}}f_1-\deg g-C,\]
	where we assume that $\deg x^{d_{(1,1)}}f_1 \geqslant \deg g+C$ since otherwise there is nothing to prove.
	
	Hence 
	\[\deg x^{d_{(1,1)}}f_1\leqslant \deg g+C\area_{I,F}(g)+C.\]
	Let $D=\max \{\deg f_i, i=1,2,\dots,s\}$, then we have
	\[ \max \{\deg h_1,\deg h_2,\dots, \deg h_s\}=\|d_{1,1}\|\leqslant \deg g+C\area_{I,F}(g)+C+D.\]
	The theorem follows immediately, since $\mathrm{diag}\ f\leqslant 2\deg f$ for all $f\in R$. The constant in the statement can be chosen to be $5D$.
\end{proof}
\begin{rems}
With \cref{degreeBound}, we can solve the ideal membership problem and the representation problem of $I$, provided by the information of $C_I(m,n)$. Let $F=\{f_1,f_2,\dots,f_s\}$ be a generating set of $I$ and $D=\max \{\deg f_i, i=1,2,\dots,s\}$. Take $g\in R$ such that $|g|=m,\deg g=n$. We just enumerate all the possible combinations of 
	\[h_1f_1+h_2f_2+\dots+h_sf_s,\]
	where $\sum_{i=1}^s |h_i|\leqslant C_{I,F}(m,n)$, $\deg h_i\leqslant n+5D C_{I,F}(m,n)$. And check if $g$ equals to one of them. If so, we have that $g\in I$ in the meantime we also solve the representation problem of $g$. Otherwise, $g\notin I$.
	
On the contrary, since both of the ideal membership problem and the representation problem are decidable, the function $C_I(m,n)$ is computable.  
\end{rems}

\section{Dehn Functions and relative Dehn functions of metabelian groups}

In this section, we will connect the ideal membership problem of an ideal over the ring of Laurent polynomials to the word problem of a finitely generated metabelian group. 

Let $G$ be a finitely presented group with a finite presentation $\mathcal P=\langle X \mid R \rangle$. Then there exists an epimorphism $\varphi: F(X) \to G$ with $\ker \varphi=\llangle R \rrangle$. A word $w$ represents the identity in $G$ if and only if it can be written as
\begin{equation}
\label{decomposition01}
	w=_{F(X)} \prod_{i=1}^l f_i^{-1}r_i{f_i} \text{ where }r_i\in R\cup R^{-1},f_i\in F(X).\tag{$*$}
\end{equation}

\begin{definition}
	The smallest possible integer $l$ over all decompositions (\ref{decomposition01}) of $w$ is called the area of $w$, denoted by $\area_\mathcal P(w)$.
\end{definition}

We then define the Dehn function with respect to the presentation $\mathcal P$ as follows.
\begin{definition}
	Let $G$ be a finitely presentable group with a finite presentation $\mathcal P=\langle X\mid R\rangle$, the Dehn function with respect to $\mathcal P$ is 
	\[\delta_\mathcal P(n)=\sup\{\area_\mathcal P(w)\mid |w|_{X}\leqslant n\},\]
where $|\cdot|_{X}$ is the word length in alphabet $X$. 
\end{definition}

It is well-known that the Dehn function does not depend on the finite presentation with respect to the asymptotic equivalence relation $\cong$ \cite{gromov1987hyperbolic}, where functions $f,g:\mathbb N\to \mathbb N$ are equivalent if there exists a constant $C$ such that 
\[C^{-1}g(n)+C^{-1}n+C \leqslant f(n) \leqslant Cg(n)+Cn+C \]
for all $n\in \mathbb N$. Therefore it is valid to speak about the Dehn function of a finitely presented group.

A generalisation of Dehn function is the Dehn function relative to a variety of groups. We use it in the case of variety of metabelian groups. Denote $M_k$ the free metabelian group on $k$ generators. The class of metabelian groups forms a variety, denoted by $\mathcal A^2$. Similar to the variety of all groups, every metabelian group $G$ generated by $k$ elements is a quotient of $M_k$ with the epimorphism $\tilde \varphi: M_k \to G$. Note that every normal subgroup of a finitely generated metabelian group is a normal closure of a finite set \cite{Hall1954}. Thus $G$ processes a finite relative presentation $\tilde P=\langle X \mid A \rangle_{\mathcal A^2}$, where the normal closure of $A$ in $M_k$ is $\ker \tilde \varphi$. With the relative presentation, the relative area and the relative Dehn function can be defined analogously. A word $w$ represents the identity if and only if it can be written as 
\begin{equation}
\label{decomposition02}
	w=_{M_k} \prod_{i=1}^l f_i^{-1}r_i{f_i}, \text{ where }r_i\in A\cup A^{-1},f_i\in M_k.\tag{$\tilde *$}
\end{equation}

The relative area, denoted by $\widetilde \area(w)$, is the smallest possible integer $l$ over all decompositions (\ref{decomposition02}) of $w$. And the relative Dehn function, denoted by $\tilde \delta_G(n)$, is defined to be 
\[\tilde \delta_{G}(n)=\sup\{\widetilde \area_\mathcal P(w)\mid |w|_{X}\leqslant n\}.\]
It is also independent of the choice of the finite relative presentations \cite{Fuh2000}. 

For each ideal $I$ of $R=\mathbb Z[x_1^{\pm 1},x_2^{\pm 1},\dots,x_k^{\pm 1}]$ generated by $\{f_1,f_2,\dots,f_s\}$, we consider a metabelian group $G_I=(R/I)\rtimes Q$ where $Q\cong \mathbb Z^k$. $Q$ acts on $R/I$ in a canonical way, i.e., $q\curvearrowright f=x^{q}\cdot f$, for $q\in Q, f\in R/I$. To simplify the notation, we will denote
$a^{\alpha_1 q_1} a^{\alpha_2 q_2} \dots a^{\alpha_s q_s}, \alpha_i\in \mathbb Z, q_i\in Q$ as
$a^{q}, \text{ where } q=\alpha_1 q_1+\alpha_2 q_2+\dots+\alpha_s q_s \in \mathbb ZQ.$

We have a correspondence between elements in the ring $R/I$ and elements in the normal subgroup $R/I$ in $G_I$. Note that $G_I$ is isomorphic to $(\mathbb Z \wr Q)/N$ where $N$ is the normal subgroup generated by $F$. We have a relative presentation of $G_I$.
\begin{equation}
\label{presentation}
	G_I=\langle a, x_1,x_2,\dots,x_k \mid [x_i,x_j], 1\leqslant i<j\leqslant k, [a^{x_i},a], 1\leqslant i \leqslant k, a^{f_l}, 1\leqslant l\leqslant s\rangle_{\mathcal A^2}. \tag{$\mathcal P$}
\end{equation}

For an element $g\in R$
\[g\in I \iff a^g=_{G_I} 1.\]
Thus the ideal membership problem of $I$ in $R$ can be reduce to the word problem in the metabelian group $G_I$.

\begin{theorem}
\label{relationship}
		Let $\tilde \delta_{\mathcal P}$ be the relative Dehn function of $G_I$ with respect to the presentation $\mathcal P$, then there exists a constant $C>0$ such that 	
		\[\tilde \delta_{\mathcal P}(Cmn) \geqslant C_{I,F}(m,n).\]
\end{theorem}

\begin{proof}
	Let $g$ be an element in the ideal $I$. Then there exists $h_i\in \mathbb R, i=1,2,\dots,s$ such that 
	\[g=\sum_{i=1}^s h_i f_i, \text{ and }\area_{I,F}(g)=\sum_{i=1}^s |h_i|.\]
	Thus we have 
	\[a^g=a^{\sum_{i=1}^s h_if_i}=\prod_{i=1}^s a^{h_i f_i}.\]
	Next we estimate the relative area of $a^g$. Note that 
	\[a^g \equiv r_1 (a^{f_{i(1)}})^{g_{1}} r_2 (a^{f_{i(2)}})^{g_{2}} \dots  r_l (a^{f_{i(l)}})^{g_{l}} r_{l+1}, i(j)\in \{1,2,\dots,s\}, g_i\in Q, r_i\in \llangle [x_i,x_j], [a^{x_i},a]\rrangle.\]
	Since $r_i$ equals $0$ in the ring $R$, then we have that 
	\[\sum_{j=1}^l g_j f_{i(j)} =_R \sum_{i=1}^s h_if_i. \]
	Since $h_i$ minimizing the area, we can deduce that 
	\[\widetilde \area\  a^g\geqslant l \geqslant \sum_{i=1}^s |h_i|=\area_{I,F}(g).\]
	Now suppose $|g|=m$ and $\deg g=n$, then the length of $a^g$ is less than or equal either to $Cm+2mn$, where $C=\max\{|a^{f_i}|\}$. Then it is easy to see that 
	\[\tilde \delta_{\mathcal P}(Cm+2mn) \geqslant C_{I,F}(m,n).\] 
\end{proof}

\begin{rems}
For the word $a^g$ in the metabelian group $G_I$, if one considers the word $w$, where $w$ is a rearrangement of conjugates of $a$ in $a^g$ and of the minimal length among all such rearrangements, the length of $w$ can be further estimated to be $m+En^k$ where $E$ is a constant depended only on $k$ \cite{wang2021subgroup}, \cite{davis2011Subgroup}. Hence we have another estimation besides the result in \cref{relationship} as follows.
\[\tilde \delta_{\mathcal P}(m+En^k) \geqslant C_{I,F}(m,n).\] 
\end{rems}

Let $\chi:Q\to \mathbb R$ is a homomorphism. We define 
\[Q_\chi =\{q\in Q \mid \chi(q)\geqslant 0\}.\]
Then a $\mathbb ZQ$-module $M$ is \emph{tame} if and only if for every homomorphism $\chi: Q\to \mathbb R$, $M$ is finitely generated either as a $\mathbb ZQ_\chi$-module or as a $\mathbb ZQ_{-\chi}$-module or both. Take the group $G$ and the ideal $A$ as above. Bieri and Strebel have shown that if $G$ is finitely presented if and only if $R/I$ is tame \cite{bieri1980valuations}.

\begin{corollary}
\label{finitelyPresented}
	Let $G_I$ defined as above and suppose $R/I$ is tame. Then, there exists a finite presentation $\mathcal P'$ of $G_I$ and a constant $D$ such that 
		\[D\delta_{G_I,\mathcal P'}(Dmn)+Dmn+D \geqslant C_I(m,n).\]
\end{corollary}

\begin{proof}
By \cite[Theorem 2.6]{wang2021subgroup}, there exists a finite presentation $\mathcal P'$ of $G_I$ such that 
\[C\delta_{G_I,\mathcal P'}(n)+Cn+C\geqslant \tilde\delta_{G_I,\mathcal P}(n)\]
for some constant $C>0$. Then the result follows from \cref{relationship}.
\end{proof}

\medskip

\bibliography{../MyLibrary}{}
\bibliographystyle{alpha}

\end{document}